\documentclass[12pt]{amsart}
\usepackage{amsfonts, amssymb, amsmath, amsthm, eucal, latexsym, array, pictex}
\usepackage{fullpage}
\usepackage{verbatim}

\usepackage{color}
\definecolor{darkblue}{rgb}{0,0,.5}
\usepackage[colorlinks=true,linkcolor=darkblue,urlcolor=violet,citecolor=magenta]{hyperref}

\numberwithin{equation}{section}

\input{cyracc.def}

\font\tencyr=wncyr10 
\font\tencyi=wncyi10 
\font\tencysc=wncysc10 

\def\rus{\tencyr\cyracc}
\def\rusi{\tencyi\cyracc}
\def\rusc{\tencysc\cyracc}

\newtheorem{thm}{Theorem}[section]
\newtheorem{lm}[thm]{Lemma} 
\newtheorem{cl}[thm]{Corollary}
\newtheorem{prop}[thm]{Proposition}

\theoremstyle{remark}
\newtheorem{ex}[thm]{Example}
\newtheorem{rmk}[thm]{Remark}

\theoremstyle{definition}
\newtheorem{df}[thm]{Definition}

\newcommand{\gt}{\mathfrak}

\newcommand{\SL}{{\rm SL}}

\newcommand{\id}{{\rm id}}
\newcommand{\ind}{{\rm ind\,}}

\newcommand{\rk}{\mathrm{rk\,}}

\newcommand{\Lie}{\mathrm{Lie\,}}
\newcommand{\Ker}{{\rm Ker\,}}

\newcommand{\g}{\tilde{\mathfrak g}}

\newcommand {\ca}{{\mathcal A}}

\newcommand {\cF}{{\mathcal F}}
\newcommand {\cj}{{\mathcal J}}
\newcommand {\cH}{{\mathcal H}}

\newcommand {\cS}{{\mathcal S}}

\newcommand {\mK}{{\mathbb K}}

\newcommand {\Z}{{\mathbb Z}}
\newcommand {\An}{{\mathbb A^n}}

\newcommand{\ard}{\rightsquigarrow}

\renewcommand{\le}{\leqslant}
\renewcommand{\ge}{\geqslant}

\newfam\eusfam%
\font\euszw=eusm10 scaled 1200%
\font\eusac=eusm7 scaled 1200%
\font\eusacc=eusm7 scaled 1000%
\textfont\eusfam=\euszw\scriptfont\eusfam=\eusac%
\scriptscriptfont\eusfam=\eusacc%
\begin{document}
\hfill {\scriptsize  February 14, 2012}
\vskip1ex

\title[Contractions of Lie-Poisson brackets]
{One-parameter contractions of Lie-Poisson brackets}
\author{Oksana Yakimova}
\address{Emmy-Noether-Zentrum, Department Mathematik, 
Friedrich-Alexander Universit\"at Erlangen-N\"urnberg}
\curraddr{Mathematisches Institut,  Friedrich-Schiller-Universit\"at Jena,
D-07737 Jena }
\email{oksana.yakimova@uni-jena.de}
\maketitle


\section{Introduction}

Let $\gt q$ be a finite-dimensional Lie algebra defined over a field $\mK$ of characteristic zero. 
Then the symmetric algebra $\cS(\gt q)=\mK[\gt q^*]$ carries a Poisson structure induced by the Lie bracket on $\gt q$. 
In this paper, we study the algebra $\cS(\gt q)^{\gt q}$ of symmetric 
invariants. By a theorem of Duflo, it is isomorphic to  the centre 
$Z{\bf U}(\gt q)$ of the universal enveloping algebra ${\bf U}(\gt q)$, 
and therefore is of much interest in representation theory. 
We can also say that  $\cS(\gt q)^{\gt q}$ coincides with the Poisson 
centre $Z\cS(\gt q)$ of $\cS(\gt q)$ (for the definition of this object see 
Section~\ref{PS}), and one can employ methods of Poisson geometry 
to investigate this algebra.  

To be more precise, the Lie algebra in question is a contraction of 
some other Lie algebra, whose symmetric invariants are well 
understood.  Already contractions of simple (non-Abelian) 
Lie algebras provide a fairy interesting and not yet completely 
explored field of research. Let $\gt f\subset\gt  q$ be a Lie subalgebra 
and $V\subset\gt q$ a complementary subspace, not necessarily 
$\gt f$-stable. Then one contracts $\gt q$ to a Lie algebra 
$\tilde{\gt q}=\gt f\ltimes V$, where $V$ is an Abelian ideal and the action of $\gt f$ on it comes from the projection ${\rm pr}_{V}: \gt q\to  V$
along $\gt f$. A more sophisticated description of contractions 
of Poisson and Lie algebras is given in Section~\ref{sec:contr}. 

Suppose that $\gt g$ is a reductive Lie algebra. 
Let $F_1,\ldots,F_\ell$ with $\ell=\rk\gt g$ be homogeneous generators of 
$Z{\cS}(\gt g)$. Then, by  \cite[Theorem 9]{ko63}, their differentials 
$d_\xi F_i$ at a point $\xi\in\gt g^*$ are linear independent if and only if 
$\dim\gt g_\xi=\ell$ for the stabiliser in the coadjoint action. 
This is known as Kostant's regularity criterion.
For an arbitrary Lie algebra $\gt q$, the notion of {\it index}, 
$\ind\gt q=\min\limits_{\xi\in\gt q^*}\dim\gt q_\xi$,
generalises the rank in the reductive case.  
A Lie algebra $\gt q$ of index $\ell$ is said to be of {\it  Kostant type}, if 
$Z\cS(\gt q)$ is freely generated 
by homogeneous polynomials  $H_1,\ldots,H_\ell$ 
such that they give Kostant's regularity criterion on $\gt q^*$. 
Set $\gt q^*_{\rm sing}:=\{\xi\in\gt q^*\mid \dim\gt q_\xi > \ind\gt q\}$.
We say that $\gt q$ has a {\it ``codim-2'' property} or satisfies a 
{\it ``codim-2'' condition}, if $\dim\gt q^*_{\rm sing}\le \dim\gt q-2$. 
The importance of this condition was first noticed in \cite{ppy} and 
\cite{Dima06}. 

The decomposition $\gt q=\gt f{\oplus} V$ induces a bi-grading on $\cS(\gt q)$.
For each homogeneous $F\in\cS(\gt q)$, let $\deg_t F$ denote its degree in $V$ 
and $F^\bullet$ the bi-homogeneous component of $F$ of bi-degrees $(\deg F-\deg_t F,\deg_t F)$ in $\gt f$ and $V$, respectively. 
In case $\gt q=\gt g$ is reductive and $\g$ is the contraction of $\gt g$ corresponding 
to the decomposition $\gt g=\gt f{\oplus}V$, a simplification of our main result, Theorem~\ref{contr-deg},
can be formulated as follows.  

Suppose that $\ind\g=\ell=\rk\gt g$. 
Then 
\begin{itemize}
\item $\sum\deg_t F_i\ge \dim V$ and the polynomials 
$F_i^\bullet$ are algebraically independent if and only if 
$\sum\deg_t F_i = \dim V$. 
\item Moreover, if the equality holds 
and the polynomials $F_i^\bullet$ generate $Z\cS(\g)$ (this can be guarantied by the 
``codim-2'' property of $\g$), then $\g$ is of Kostant type. 
\end{itemize}
The proof of Theorem~\ref{contr-deg} relies on Lemma~\ref{sleva-sprava}, 
a statement about Poisson brackets in the algebraic setting, and 
good behaviour of the Poisson tensor under contractions.   
The resulting algebras $\g$ are non-reductive and there is no general method for
describing their symmetric invariants. 

Note that 
Theorem~\ref{contr-deg} is stated and proved for arbitrary polynomial Poisson algebras that are not 
necessary symmetric algebras of any finite-dimensional $\gt q$. We do not consider applications of the more general version 
in this paper, but hope to explore this subject in the (near) future. 

Two types of contractions $\gt g\ard\g$ are studied here. 
In both cases it is assumed that 
the ground filed $\mK$ is algebraically closed. The first contraction  
comes from a $\Z_2$-grading (or symmetric decomposition) 
$\gt g=\gt g_0{\oplus}\gt g_1$ of $\gt g$. It was conjectured by 
D.\,Panyushev \cite{Dima06}, that in this setting $\cS(\g)^{\g}$ is a polynomial algebra in $\ell$ 
variables. As was shown in \cite{Dima06}, $\ind\g=\ell$ and $\g$ has the ``codim-2'' property. 
Also for many $\Z_2$-gradings the polynomiality of $\cS(\g)^{\g}$ was established in that paper 
of Panyushev. For four of the remaining cases, 
we construct homogeneous generators $F_i$ such that $\sum\deg_t F_i\le \dim\gt g_1$. Since also $\sum\deg_t F_i\ge\dim \gt g_1$ by Theorem~\ref{contr-deg}, we get the equality 
$\sum\deg_t F_i = \dim\gt g_1$
and thereby 
prove that their components $F_i^\bullet$ freely generate $\cS(\g)^{\g}$.   
This line of argument resemblances proofs of \cite[Theorems~4.2\&4.4]{ppy}. 
Our result confirms a weaker version of Panyushev's conjecture. If the restriction homomorphism 
$\mK[\gt g]^{\gt g}\to\mK[\gt g_1]^{\gt g_0}$ is surjective, then  
 $\cS(\g)^{\g}$ is a polynomial algebra in $\ell$ 
variables, see Theorem~\ref{goodpairs}.

The second contraction  $\gt g\ard\g$ was recently introduced by E.\,Feigin \cite{feigin1} and 
for the resulting Lie algebra, 
$\g$-invariants in $\cS(\g)$ and $\mK[\g]$ were studied in \cite{py-feig}. 
Here the decomposition is $\gt g=\gt b{\oplus}\gt n^-$, where $\gt b=\Lie B$ is a Borel subalgebra and 
$\gt n^{-}$ is the nilpotent radical of an opposite Borel. 
Complementing and relying on results of \cite{py-feig}, we show that $\g$ is of Kostant 
type (Lemma~\ref{feig-kost}), compute its fundamental semi-invaraint (see Definition~\ref{def-fund}
and Theorem~\ref{fund-f}), and prove that the subalgebra $\cS(\g)_{\rm si}\subset\cS(\g)$ generated 
by semi-invariants of $\g$ (Definition~\ref{semi}) is a polynomial algebra in $2\ell$ variables, 
Theorem~\ref{semi-f}.  If $\gt g$ is not of type $A$, then $\g$ does not have 
the ``codim-2" property. However, the quotient map 
$\mK[\g^*]\to \mK[\g^*]^{\g}$ is equidimensional and ${\bf U}(\g)$ is a free 
$Z{\bf U}(\g)$-module \cite{py-feig}. 

Finally, Section~\ref{orbital} contains a few observations related to subregular orbital 
varieties $\gt D_i$, which are linear subspaces of $\gt n$ of codimension $1$ forming 
the complement of the open $B$-orbit in $\gt n$. In particular, in Proposition~\ref{Ab}, 
we list all $\gt D_i$ such that the stabiliser $B_x$ is Abelian for a generic $x\in\gt D_i$.  

\noindent
{\bf Acknowledgements.}  Main part of this paper was written
in Bonn, while I was on  medical leave from FAU Erlangen-N\"urnberg. 
It is a pleasure to express my gratitude to Doctor J\"urgen Martin
(Facharzt f\"ur Chirurgie)
for his kindness and the outmost care and attention, with which he handled my injured hand.

\section{Generalities on polynomial Poisson structures}\label{PS}

In this section, we recall a rather important equality 
in Poisson algebras, which has an origin in
mathematical physics  \cite{or}.

Let $\mK$ be a field of characteristic zero and
$\mathbb A^n=\mathbb A^n_{\mathbb K}$ the $n$-dimensional
affine space with the algebra of regular functions ${\mathcal A}=\mathbb K[x_1,\ldots,x_n]$. 
Let $\Omega$ be the algebra of regular, i.e., with polynomial coefficients,
 differential forms on $\An$ and 
$W$  the algebra of  derivations of 
$\mathcal A$. Both are free $\mathcal A$-modules with bases
consisting of skew-monomials in $dx_i$ and $\partial_i=\partial_{x_i}$, respectively. In other words, 
$W$ is a graded skew-symmetric algebra generated by polynomial vector fields on $\An$. 
We identify $\Omega^0$ with $\mathcal A$ and regard
$\Omega^1$ as the $\mathcal A$-module of global sections of the
cotangent bundle $T^*\mathbb A^n$. 
Let $W^1$ be an $\ca$-module  generated by  $\partial_i$ with $1\le i\le n$.
We view the exterior powers $\Omega^k=\Lambda_{\ca}^k\Omega^1$ and 
$W^k:=\Lambda_{\ca}^kW^1$ as dual $\ca$-modules by extending the canonical non-degenerate $\ca$-pairing
$dx_i(\partial_j)=\delta_{ij}$.

Let $\omega=dx_1\wedge\ldots\wedge dx_n$ be the  volume form. 
If $f$ and $g$ are elements of $\Omega^k$ and $\Omega^{n-k}$,
respectively, then $f\wedge g=a\omega$ with $a\in\ca$.
We will say that in this situation $a=(f\wedge g)/{\omega}$ and 
$f/\omega$ is an element of $(\Omega^{n-k})^*$ such that 
$(f/\omega)(g)=a$. 
This defines an $\ca$-linear map 
$$
\frac{1}{\omega}: \Omega^k \to (\Omega^{n-k})^*\cong W^{n-k}.
$$

Suppose that ${\ca}$ possesses a Poisson structure 
$\{\,\,,\,\}\colon\,{\mathcal A}\times{\mathcal A}\to {\mathcal A}$ and let $\pi$ denote the corresponding {\it Poisson tensor (bivector)}, the element of $\mathrm{Hom}_{\mathcal A}(\Omega^2,{\mathcal A})$ 
satisfying $\pi(df\wedge dg)=\{f,g\}$
for all $f,g\in{\mathcal A}$. 
(It is {\it not} assumed that the coefficients of $\pi$ are linear functions.) 
In view of the duality between forms and vector fields,
we may regard $\pi$ as an element of $W^2$.
For $\xi\in \mathbb A^n$,
$\pi_\xi$ can be viewed as a skew-symmetric matrix with entries 
$\{x_i,x_j\}(\xi)$.
The {\it index} of the Poisson algebra ${\mathcal A}$, denoted
$\ind\mathcal A$, is defined as
$$
\ind\mathcal A:=n-\rk\pi, \ \text{where } \ \rk\pi=\max_{\xi\in \mathbb A^n}\,\rk\pi_\xi.
$$
An element $a\in\ca$ is said to be {\it central}, if $\{a,\ca\}=0$.
Correspondingly, the set $Z{\mathcal A}=Z(\ca,\pi)$ of all central elements is called  the 
{\it Poisson centre} of ${\mathcal A}$.

Set $\mathrm{Sing}\,\pi:=\,\{\xi\in\mathbb A^n\,|\,\,
\rk\pi_\xi<\rk\pi\}$. Clearly, $\mathrm{Sing}\, \pi$ is
a proper Zariski closed subset of $\mathbb A^n$. By definition, 
$\pi(da\wedge db)=0$ for all $a\in Z{\mathcal A}$
and all $b\in{\mathcal A}$. Hence the linear subspace $\{d_\xi a\mid a\in Z{\mathcal A}\}$ lies in the kernel of $\pi_\xi$ and
we have 
$$ 
\mathrm{tr.\,deg}_{\mathbb K}\,Z\ca\le\ind\mathcal A.
$$

For $g_1,\ldots, g_m\in\mathcal A,$ the {\it Jacobian locus} ${\cj}(g_1,\ldots,g_m)$ consists of all
$\xi\in{\mathbb A}^n$ such that the differentials 
$d_\xi g_1,\ldots,d_\xi g_m$ are linearly dependent. 
In other words, $\xi\in {\cj}(g_1,\ldots,g_m)$ if and only if 
$(dg_1\wedge\ldots\wedge dg_m)_\xi=0$.
The set ${\mathcal J}(g_1,\ldots,g_m)$ is Zariski closed in $\mathbb A^n$ and it
coincides with $\mathbb A^n$ if and only if $g_1,\ldots, g_m$ are
algebraically dependent.

Given $k\in\mathbb N$ we let
$$
\Lambda^k\pi:=\,\underbrace{\pi\wedge\pi\wedge\ldots\wedge \pi}_{k \ \scriptstyle{\mathrm{factors}}}\,,
$$ 
be an element of
$W^{2k}$. Note that $\Lambda^k\pi\ne 0$ if and only if 
$\pi_\xi$ contains a non-zero $2k{\times}2k$-minor for some $\xi\in\An$.
Therefore $\Lambda^k\pi =0$ for $k>(\rk\pi)/2$ and $\Lambda^k\pi\ne 0$
for $k\le(\rk\pi)/2$. 
 The following statement can be extracted from 
the proofs of  \cite[Theorem 3.1]{or}, \cite[Theorem~1.2]{ppy}, 
\cite[Theorem~1.2]{Dima06}.

\begin{lm} \label{sleva-sprava}
Let ${\mathcal A}=\mathbb K[x_1,\ldots,x_n]$ be a
Poisson algebra of index $\ell$ and  let 
$\{F_1,\ldots,F_\ell\}\subset Z{\mathcal A}$ be a set of algebraically independent 
polynomials. Then there are coprime $q_1,q_2\in\ca\setminus\{0\}$ such that
$$
q_1\frac{dF_1\wedge\ldots\wedge dF_\ell}{\omega}=q_2 \Lambda^{(n-\ell)/2}\pi\,.
$$
\end{lm}
\begin{proof} 
Set ${\mathcal F}=dF_1\wedge\ldots\wedge dF_\ell$.
Because of the  inequality: ${\rm tr.\,deg}\,Z(\ca)\le \ell$, 
the polynomials $F_1,\ldots,F_{\ell}$, and $F$ are algebraically  dependent
for each $F\in Z\ca$ and therefore ${\mathcal F}\wedge F=0$. Clearly 
$\pi(dF,\,.\,)=0$ and hence $\Lambda^{n-\ell/2}\pi$ is zero
on $dF\wedge\Omega^{n-\ell-1}$.

Changing the ordering of the coordinates, we may assume that
${\mathcal F}\wedge dx_1\wedge\ldots\wedge dx_{n-\ell}\ne 0$.
Let $\xi\in\An$ be such that $\rk\pi_\xi=n-\ell$ and the elements 
$d_\xi F_i$ together with $\{dx_j\mid j\le n-\ell\}$ form a basis of 
$T^*_\xi\An$. Then 
$$
\Lambda^{n-\ell}(T^*_\xi\An)=\Ker\left(\frac{{\mathcal F}}{\omega}\right)_{\!\!\xi}\oplus{\mK} (dx_1\wedge\ldots\wedge dx_{n-\ell}).
$$
Note that
here the kernel of $\left({\mathcal F}/{\omega}\right)_\xi$ lies also in the kernel
of $\Lambda^{(n-\ell)/2}\pi_\xi$. Next $\Lambda^{(n-\ell)/2}\pi_\xi\ne 0$. 
Therefore $\left({\mathcal F}/{\omega}\right)_\xi$ is proportional to 
$\Lambda^{(n-\ell)/2}\pi_\xi$. We can conclude that  
${\mathcal F}/{\omega}$ and $\Lambda^{(n-\ell)/2}\pi$
are  proportional on an open subset of $\An$.
It follows that there exist non-zero coprime $q_1,q_2\in{\mathcal A}$ such that
$q_1({\mathcal F}/{\omega})=q_2\Lambda^{(n-\ell)/2}\pi$.
\end{proof}

Of particular interest are situations where $q_1,q_2\in\mK$ for 
$q_1,q_2$ as above. This can be guarantied by ``codim-2" conditions, 
see e.g. \cite[Theorem~1.2]{ppy}. If $\dim\mathrm{Sing}\,\pi\le n-2$, 
then $q_1$ must be a scalar. If $\dim\cj(F_1,\ldots,F_\ell)\le n-2$, 
then $q_2$ must be a scalar. 

In case $\pi$ is homogeneous, i.e., all the 
(polynomial) coefficients of $\pi$ are of the same degree, we can say that 
$\deg\Lambda^k\pi=k\deg\pi$. Suppose that $\pi$ and all the $F_i$'s are 
homogeneous. Then $q_1,q_2$ are also homogeneous and
$$
\deg q_1-\ell+\sum_{i=1}^{\ell}\deg F_i=\deg q_2+\frac{n-\ell}{2}\deg\pi\,.
$$
If, for example, $2\sum\deg F_i=2\ell+(n-\ell)\deg\pi$, then 
$\deg q_1=\deg q_2$  and knowing that $q_1$ is constant, we also know that
$q_2$ is a constant. 
 
Poisson tensors of degree $1$ correspond to finite-dimensional Lie algebras
$\gt q$ over $\mK$. In this case $\An=\gt q^*$ is the dual space of 
an $n$-dimensional Lie algebra $\gt q$ and $\ca=\cS(\gt q)=\mK[\gt q^*]$
is the symmetric algebra of $\gt q$. Set $\ind\gt q=\ind \cS(\gt q)$ and 
$\gt q^*_{\rm sing}=\mathrm{Sing}\,\pi$. Note that 
$$
\ind\gt q=\min_{\gamma\in\gt q^*} \dim\gt q_\gamma=\dim\gt q_\alpha
\mbox{ for all } \alpha\in \gt q^*_{\rm reg}=\gt q^*\setminus\gt q^*_{\rm sing}\,.
$$ 
It is also worth mentioning that $Z\cS(\gt q)=\cS(\gt q)^{\gt q}$.

Suppose that $\gt g$ is a non-Abelian reductive Lie algebra, then $\ind\gt g=\rk\gt g$,
the algebra of symmetric invariants  $Z{\cS}(\gt g)$ is freely generated 
by homogeneous polynomials $F_1,\ldots,F_\ell$ with $\ell=\rk\gt g$, and
$2\sum\deg F_i=n+\ell$. Moreover, $\dim\gt g^*_{\rm sing}=n-3$.
Therefore, after a suitable renormalisation, 
\begin{equation}\label{Kostant}
\frac{dF_1\wedge\ldots\wedge dF_\ell}{\omega}= \Lambda^{(n-\ell)/2}\pi\,.
\end{equation}
This is known as Kostant's regularity criterion: $x\in\gt g^*_{\rm reg}$ if and 
only if the differentials $d_xF_i$ are linear independent, \cite[Theorem 9]{ko63}. 

\begin{df}\label{Kostant-def}
Equation~\eqref{Kostant}  is called the {\it  Kostant equality} and 
we will say that a Poisson algebra $\ca$ (or a Lie algebra $\gt q$)
is of {\it  Kostant type}, if $Z\ca$ is generated by $\ell$ 
polynomials 
satisfying the Kostant equality.
\end{df}

Apart from reductive and Abelian Lie algebras, 
examples of  Lie algebras of Kostant type are provided by 
the centralisers $\gt g_e$ of nilpotent elements in $\gt{sl}_m$ and $\gt{sp}_{2m}$ \cite{ppy}, truncated seaweed (biparabolic) subalgebras of  $\gt{sl}_m$ and $\gt{sp}_{2m}$ \cite{j}, and 
semi-direct products related to symmetric decompositions $\gt g=\gt g_0{\oplus}\gt g_1$, see \cite{Dima06} and Section~\ref{sec:symm} here. 

Let $(e,h,f)$ be an $\gt{sl}_2$-triple in $\gt g=\Lie G$. Then ${\bf S}_e=e+\gt g_f$
is the {\it Slodowy slice} of $Ge$ at $e$ and $\mK[{\bf S}_e]$ inherits a Poisson bracket from $\gt g(\cong \gt g^*)$. These Poisson algebras are of Kostant type and they are 
associated graded algebras of the finite $W$-algebras, see \cite{sasha} and \cite[Section~2]{ppy}. 
Note that all mentioned above Poisson and Lie algebras of Kostant type have the ``codim-2" property.

\section{Contractions  of Poisson tensors}\label{sec:contr}

We begin with a definition of a contraction in the Lie algebra setting. 
Let $\gt q$ be a Lie algebra, $\gt f\subset\gt q$ a 
Lie subalgebra, and $V\subset\gt q$ a complimentary 
(to $\gt f$) subspace. We do not require $V$ to be 
$\gt f$-stable. For each $t\in \mK^{^\times}$, let 
$\varphi_t:\gt q\to\gt q$ be a linear map 
multiplying vectors in $V$ by $t$ and vectors in 
$\gt f$ by $1$.  These automorphisms form a one-parameter subgroup 
in ${\rm GL}(\gt q)$.
Each $\varphi_t$ defines 
a new Lie algebra structure $[\,\,,\,]_t$ on the same vector space 
$\gt q$. Let ${\rm pr}_{\gt f}$ and ${\rm pr}_{V}$ be the projection 
on $\gt f$ and $V$, respectively.  Then
$$
[\xi,\eta]_t=[\xi,\eta], \ [\xi,v]_t=t{\rm pr}_{\gt f}([\xi,v])+{\rm pr}_V([\xi,v]), \
[v,w]_t=t^2{\rm pr}_{\gt f}([v,w])+t{\rm pr}_V([v,w]), 
$$
for
$\xi,\eta\in\gt f$, $v,w\in V$.
We can pass to the limit  $\lim_{t\to 0}[\,\,,\,]_t$ and get yet another 
Lie algebra structure on $\gt q$. 
Let $\tilde{\gt q}$ denote this 
contraction of  $\gt q$. Then $\tilde{\gt q}=\gt f\ltimes V$, 
where $V$ is an Abelian ideal of $\tilde{\gt q}$ and the action 
of $\gt f$ on $V$ is given by ${\rm pr}_V$. (The reader feeling uncomfortable with taking the limit, although it 
can be defined in a purely algebraic setting, may assume that $t$ takes values in $\mathbb Q\subset\mK$.) 

In a coordinate free way, 
the $t$-commutator $[\,\,,\,]_t$ or the $t$-Poisson bracket  
$\{\,\,,\,\}_t$ is defined by 
$$
\{x,y\}_t=\varphi_t^{-1}(\{\varphi_t(x),\varphi_t(y)\})
$$
for $x,y\in\gt q$. Extending $\varphi_t$ to the symmetric algebra 
$\cS(\gt q)$ as well as to $W$ and $\Omega$, one can say that
$\pi_t=\varphi_t^{-1}(\pi)$. Let $\gt q_t$ stand for the Lie algebra 
correposnding to $\pi_t$. Then the Poisson centre of 
$\cS(\gt q_t)$ 
equals $\varphi_t^{-1}(Z\cS(\gt q))$. 

For $H\in{\mathcal S}(\gt q)$, let $\deg_t H$ be the degree in $t$ of 
$\varphi_t(H)$. This means that  
$\varphi_t(H)=t^dH_d+t^{d-1}H_{d-1}+\ldots+H_0$, where $d=\deg_t H$,
$H_i\in\cS(\gt q)$, and $H_d\ne 0$. We will say that $H^\bullet:=H_d$ 
is the highest ($t$-) component of $H$.

\begin{lm}\label{bullet}
If $H\in Z\cS(\gt q)$, then $H^\bullet$ is a central element in $\cS(\tilde{\gt q})$.
\end{lm}
\begin{proof}
Since $H\in\cS(\gt q)^{\gt q}$, 
its preimage $\varphi_t^{-1}(H)$ is a central element in $\cS(\gt q_t)$, which
one can write as $\varphi_t^{-1}(H)=t^{-d}H_d+t^{1-d}H_{d-1}+\ldots t^{-1}H_1+H_0$. 
Multiplying it by $t^d$, we get that 
$\sum\limits_{j=0}^{d} t^{d-j}H_j$ is also a central element in $\cS(\gt q_t)$. 
Passing to the limit at $t\to 0$, one obtains 
that $H_d=H^\bullet$ is an element of $Z \cS(\tilde{\gt q})$.
\end{proof}

The automorphism $\varphi_t: \gt q\to \gt q$ does not need to be of degree $1$ in $t$
as well as the Poisson tensor $\pi$ does not need to be linear. 
We can consider a one-parameter family of linear automorphisms of $\An$ and the corresponding deformation of Poisson structures on it. The only important thing as that there exists a limit $\lim_{t\to 0}\pi_t$.
In order to be consistent with the Lie algebra case, 
we identify $\An$ with $\mK^n$. Let $\varphi$ be a $\mK$-linear automorphism of the dual space 
$(\mK^n)^*$. Then $\varphi$ extends to $\mK$-linear automorphisms 
of $\An$, $\ca=\mK[\An]$, $W$, and $\Omega$. 

\begin{df}\label{contr-pol}
Let $\pi$ be a polynomial Poisson tensor on $\An\cong\mK^n$.
Suppose that we have a family of automorphisms $\varphi_t$ 
given by  a regular map $\mK^{^\times}\to {\rm GL}((\mK^n)^*)$ and that 
the formal expression of $\pi_t:=\varphi_t^{-1}(\pi)$ is an element of  $W^2[t]$.
Then
$\tilde\pi:=\lim_{t\to 0}\pi_t$ is called a {\it contraction} of $\pi$.  
For each  $H\in\ca$, we define its {\it highest ($t$-) component} 
as a non-zero polynomial $H^\bullet$ such that 
$H^\bullet =\lim_{t\to 0} t^d\varphi_t^{-1}(H)$ for some $d=:\deg_t H$.  
(One readily sees the uniqueness of this $d$.) 
\end{df}

\begin{lm}\label{bullet-gen}
If $\tilde\pi$ is a contraction of $\pi$, then $\tilde\pi$ is  again a Poisson tensor and 
for each 
$H\in Z(\ca,\pi)$, the polynomial $H^\bullet$ is an 
element of $Z(\ca,\tilde\pi)$. 
\end{lm}
\begin{proof}
An element $R\in W^2$ is a Poisson tensor if and only if $[R,R]=0$.
(This is a way to state the Jacobi identity.)  Since 
$[\pi_t,\pi_t]=0$ for all non-zero $t$ and $\pi_t\in W^2[t]$, we have
$\pi_t=\tilde\pi+tR$, where $R\in W^2[t]$,
and $0=[\pi_t,\pi_t]=[\tilde\pi,\tilde\pi]+t\tilde R$ with 
$\tilde R\in W^2[t]$. Therefore $[\tilde\pi,\tilde\pi]=0$.

In order to prove the second statement one repeats 
the argument of  Lemma~\ref{bullet}.
\end{proof}

\begin{ex}\label{zm} Suppose we have a decomposition 
$\gt q=V_0{\oplus}V_1{\oplus}\ldots{\oplus}V_{m-1}$, 
where $V_0$ is a subalgebra and in general 
$[V_i,V_j]\subset \bigoplus\limits_{k\le i+j} V_k$. 
Then one can define $\varphi_t:\gt q\to\gt q$ by 
setting $\varphi|_{V_j}=t^j\id$ and consider Lie 
algebra structures $[\,\,,\,]_t$ on $\gt q$. 
Clearly, there is a limit at $t\to 0$ and the resulting Lie 
algebra $\tilde{\gt q}$ has a $\mathbb Z$-grading with 
at most $m$ non-zero components. 
\end{ex}  
 
Contractions of Lie algebras as in Example~\ref{zm} were studied by 
Panyushev \cite{Dima08}.
 

\begin{df}\label{good-gs}
Let $\ell=\ind\gt q$.
We say that a set 
$\{H_1,\ldots,H_\ell\}\subset {\mathcal S}(\gt q)^{\gt q}$ 
is a {\it good generating system} with respect to
a contraction $[\,\,,\,]_t\ard [\,\,,\,]_{\tilde{\gt q}}$
if the polynomials $H_i$ generate ${\mathcal S}(\gt q)^{\gt q}$ 
and their highest components $H_i^\bullet$ are algebraically independent.
%
%
\end{df}

Let $(Z{\cS}(\gt q))^\bullet$ be the algebra of highest components 
of $Z\cS(\gt q)$, i.e., this is an algebra generated by $H^\bullet$
with $H\in Z\cS(\gt q)$.

\begin{lm}\label{h-gen}
If $\{H_1,\ldots,H_\ell\}\subset {\mathcal S}(\gt q)^{\gt q}$ 
is a  good generating system, then $H_i^\bullet$ generate 
$(Z{\cS}(\gt q))^\bullet$.
\end{lm}
\begin{proof}
Each non-zero element $g\in Z\cS(\gt q)$ can be expressed as a polynomial
$P$ in $H_i$. Suppose that $P$ is a sum 
$P=\sum_{\bar s} a_{\bar s} H_1^{s_1}\ldots H_\ell^{s_\ell}$ over 
some $\bar s\in \Z_{\ge 0}^\ell$. Define $\tilde P$ as a sum 
of those monomials (with the coefficients $a_{\bar s}$), where the degree in $t$,
$s_1\deg_t H_1+\ldots s_\ell\deg_t H_\ell$, is maximal. 
Then $\tilde P(H_1^\bullet,\ldots,H_\ell^\bullet)$ is a non-zero polynomial,  
because the elements $H_i^\bullet$ are algebraically independent, and it 
equals $g^\bullet$ by the construction.
\end{proof}

\smallskip 

\subsection{Contractions and the Kostant equality} 

\begin{ex}\label{sl2-contr}
Suppose that $\gt q=\gt{sl}_2$ and a contraction $\gt q\ard\tilde{\gt q}$
is defined by a 
decomposition $\gt{sl}_2=\gt{so}_2{\oplus} V$, where 
$V$ is an $\gt{so}_2$-invariant complement. 
In a standard basis $\{e,h,f\}$ the automorphism $\varphi_t$
multiplies $e$ and $f$ by $t$. In the basis
$\{e/t,h,f/t\}$ the Poisson tensor $\pi_t$ is given by the same formula
as $\pi$ in the original basis. Therefore we have 
$$
\frac{dF}{d(e/t)\wedge dh\wedge d(f/t)}=h\partial_{e/t}{\wedge}\partial_{f/t}+
  2\frac{e}{t}\partial_h{\wedge}\partial_{e/t}+2\frac{f}{t}\partial_{f/t}{\wedge}\partial_h,
$$
where $F$ is a suitably normalised invariant of degree $2$,
explicitly $F=-\frac{h^2}{2}-\frac{2ef}{t^2}$.
After removing $t$ from denominators, the  above equality modifies to
$$
\frac{-t^2hdh-2fde-2edf}{de\wedge dh\wedge df}=t^2h\partial_{e}{\wedge}\partial_{f}+
  2{e}\partial_h{\wedge}\partial_{e}+2{f}\partial_{f}{\wedge}\partial_h.
$$
In particular, for $\tilde{\gt q}$, we have $dF^\bullet/\omega=\tilde\pi$.
\end{ex}

Example~\ref{sl2-contr} illustrates a general phenomenon. 
Let $D_t$ be the degree in $t$ of the determinant
of the map $\varphi_t:(\mK^n)^*\to(\mK^n)^*$, where 
$\mK^n$ is identified with $\An$. In case of  a linear (in $t$) 
contraction of a Lie algebra $\gt q$, we have $D_t=\dim V$ and
$\varphi_t$ multiplies the canonical volume form 
$\omega$ by $t^{D_t}$.

\begin{thm}\label{contr-deg}
Suppose we have a contraction $\pi_t\ard \tilde\pi$ of
a Poisson structure $\pi$ on $\An\cong\mK^n$ given by a family of linear 
automorphisms $\varphi_t:(\mK^n)^*\to(\mK^n)^*$ with the determinants
$t^{D_t}$. 
Suppose further that $\ind\ca=\ell$ and it stays the same under
the contraction.  If the Kostant equality holds for a set of polynomials 
$F_1,\ldots,F_\ell\in Z(\ca,\pi)$,
then 
\begin{itemize}
\item[({\sf i})] \  $\sum\deg_t F_i\ge D_t$, moreover, if $\sum\deg_t F_i > D_t$, then 
$F_i^\bullet$ are algebraically dependent;
\item[({\sf ii})] \ if $\sum\deg_t F_i =D_t$, then $F_i^\bullet$ are algebraically independent
and satisfy the Kostant equality with $\tilde\pi$;
\item[({\sf iii})] \ if we have an equality in ({\sf i}),  
$\dim\mathrm{Sing}\,\tilde\pi\le n-2$, and each $F_i^\bullet$ 
is a homogeneous polynomial, 
then $F_i^\bullet$ generate $Z(\ca,\tilde\pi)$.
\end{itemize}
\end{thm}
\begin{proof}
We are contracting, so to say, both sides in the Kostant equality.
For each non-zero $t$, we have
$$
\frac{d\varphi_t^{-1}(F_1)\wedge\ldots\wedge d\varphi_t^{-1}(F_\ell)}{\varphi_t^{-1}(\omega)}=\Lambda^{(n-\ell)/2}\pi_t
$$
and therefore
$$
\frac{t^{D_t}d\varphi_t^{-1}(F_1)\wedge\ldots\wedge d\varphi_t^{-1}(F_\ell)}{\omega}=\Lambda^{(n-\ell)/2}\tilde\pi+tR\,,
$$
where $R\in W^{n-\ell}[t]$.

If $\sum\deg_t F_i < D_t$, then taking the limit at $t\to 0$, we get zero 
on the left hand side. Since index remains the same under this contration, 
$\Lambda^{(n-\ell)/2}\tilde\pi\ne 0$, 
and this proves the inequality $\sum\deg_t F_i\ge D_t$.  Further, if  
$\sum\deg_t F_i > D_t$, then 
$t^{D_t}(dF_1^\bullet\wedge\ldots\wedge dF_\ell^\bullet)$ is either zero or tends to 
infinity as $t$ tends to zero and therefore 
$dF_1^\bullet\wedge\ldots\wedge dF_\ell^\bullet$ must be zero.
This completes the proof of part ({\sf i}).

The equality $\sum\deg_t F_i = D_t$ implies that the left hand side tends 
to $(dF_1^\bullet\wedge\ldots\wedge dF_\ell^\bullet)/\omega$ as $t$ tends 
to zero. Therefore these highest components are algebraically independent 
and indeed satisfy the Kostant equality with $\Lambda^{(n-\ell)/2}\tilde\pi$
on the right hand side.

Part ({\sf ii}) implies that $\cj(F_1^\bullet,\ldots,F_\ell^\bullet)$ 
is equal to $\mathrm{Sing}\,\tilde\pi$. Thus, if the conditions in ({\sf iii}) 
are satisfied, then the Jacobian locus of $F_i^\bullet$ has dimension at most
$n-2$. Since $\mathrm{tr.\,deg}\,Z(\ca,\tilde\pi)\le\ell$, for each 
$F\in Z(\ca,\tilde\pi)$, the polynomials $F_1^\bullet,\ldots,F_\ell^\bullet$, and $F$ are algebraically dependent. The assumption that each $F_i^\bullet$ is homogeneous, allows us to use a characteristic zero version of  Skryabin's result,
see \cite[Theorem~1.1]{ppy}, which states that in this situation $F$ lies in the subalgebra generated by $F_i^\bullet$.
\end{proof}

\section{Symmetric invariants of $\Z_2$-contractions}\label{sec:symm}

Let $G$ be a connected reductive algebraic group defined over 
$\mK$. Suppose that $\mK=\overline{\mK}$. Set $\gt g=\Lie G$. 
Let $\sigma$ be an involution (automorphism of oder $2$) 
of $G$.  On the Lie algebra level 
$\sigma$ induces a $\mathbb Z_2$-grading 
$\gt g=\gt g_0\oplus\gt g_1$, where 
$\gt g_0=\gt g^{\sigma}=\Lie G_0$ and 
$G_0:=G^{\sigma}$ is the subgroup of $\sigma$-invariant points. 
In this context, $G_0$ is said to be 
a {\it symmetric subgroup}, $G/G_0$ a {\it symmetric space} and 
$(\gt g,\gt g_0)$ a {\it symmetric pair}.
One can contract $\gt g$ to a semidirect product
$\g=\gt g_0\ltimes\gt g_1$, where $\gt g_1$ becomes an Abelian ideal,
in the same way as described in Section~\ref{sec:contr}. 
We will call the resulting Lie algebra, $\g$, a  {\it $\mathbb Z_2$-contraction}
of $\gt g$. In this section, our main objects of interest are $\mathbb Z_2$-contractions of simple (non-Abelian) Lie  algebras. 

Set $\ell=\ind\gt g=\rk\gt g$.
By \cite[Proposition 2.5]{Dima06}, $\ind\g=\ell$ for a
$\Z_2$-contraction of a reductive Lie algebra. It was also conjectured 
in \cite{Dima06} that $\cS(\g)^{\g}$ is a polynomial algebra 
in $\ell$ variables. In would be sufficient to prove the conjecture 
for symmetric pairs with simple $\gt g$. For many pairs 
it was already proved in \cite{Dima06}. Here we consider 4 of the remaining ones. This does not cover all them and does not prove Panyushev's conjecture.

\begin{prop}[{\cite[Section~6]{Dima06}}]\label{remaining}
Suppose that $\gt g$ is a simple non-Abelian Lie algebra.
Then 
all  symmetric pairs $(\gt g,\gt g_0)$ such that the polynomiality 
of ${\mathcal S}(\tilde{\gt g})^{\tilde{\gt g}}$
is not established yet are listed below. 

\vskip0.3ex
\noindent
{\bf \underline{Exceptional Lie algebras}:} 
\begin{itemize} 
\item[$\bullet$] \
$(E_6, F_4)$, $(E_7,E_6{\oplus}\mK)$,  $(E_8,E_7{\oplus}\gt{sl}_2)$, and
$(E_6,\gt{so}_{10}{\oplus}\gt{so}_2)$;  
\item[$\bullet$] \ $(E_7, \gt{so}_{12}{\oplus}\gt{sl}_2)$. 
\end{itemize}

\vskip0.3ex
\noindent
{\bf \underline{Classical Lie algebras}:}
\begin{itemize}
\item[$\bullet$] $(\gt{sp}_{2n+2m},\gt{sp}_{2n}{\oplus}\gt{sp}_{2m})$ for 
$n\ge m$;
\item[$\bullet$] $(\gt{so}_{2\ell},\gt{gl}_\ell)$; 
\item[$\bullet$] $(\gt{sl}_{2n},\gt{sp}_{2n})$.
\end{itemize}
\end{prop}

The first 4 exceptional symmetric pairs are collected in one item, because 
there are no good generating systems in ${\mathcal S}(\gt g)^{\gt g}$ 
with respect to the corresponding $\Z_2$-contractions, see \cite[Remark~4.3]{Dima06}. Moreover, it is quite possible that 
the algebra of symmetric invariants is not freely generated for these $\tilde{\gt g}$.
These are precisely the symmetric pairs such that the restriction homomorphism 
$\mK[\gt g]^G\to \mK[\gt g_1]^{G_0}$ is not surjective \cite{Helg}.

According to  \cite[Theorem~3.3.]{Dima06}, the Lie algebra $\g$ always possesses the ``codim-2" property,
$\dim\mathrm{Sing}\,\tilde\pi\le \dim\g-2$. For the pair $(E_7, \gt{so}_{12}{\oplus}\gt{sl}_2)$ and the three classical series 
listed in Proposition~\ref{remaining}, we will construct homogeneous generators 
$F_i\in\cS(\gt g)^{\gt g}$ such that $\sum\deg_tF_i\le\dim\gt g_1$ and using Theorem~\ref{contr-deg} prove that 
Panyushev's conjecture holds for them. 


For each element $x\in\gt g_1$, we let $\gt g_{i,x}=\gt g_x\cap\gt g_i$ denote
its centraliser in $\gt g_i$ ($i=0,1$).
Let $\mathfrak{c}\subset\gt g_1$ be a maximal (Abelian) subalgebra
consisting of semisimple elements. Any such subalgebra is
called a {\it Cartan subspace} of $\gt g_1$. 
Let $\gt l=\gt g_{0,\gt c}$ be the centraliser of $\gt c$ in $\gt g_0$. 
We will need a few facts that 
can be found in e.g.  \cite[Thm.\,1\&Prop.\,8]{kr}.
First, all Cartan subspaces are $G_0$-conjugate. Second, 
$\gt l=\gt g_{0,s}$ for a generic $s\in\gt c$ and therefore it is 
a reductive subalgebra. And finally, 
$G_0\gt c$ is a dense subset of $\gt g_1$. 

Let  $L:=(G_{0,\gt c})^{\circ}$ be the connected component of the identity 
of $G_{0,\gt c}=\{g\in G_0\mid gs=s \text{ for all } s\in\gt c\}$. 
Using the Killing form, we identify
$\gt g\cong\gt g^*$,
$\gt g_1\cong\gt g_1^*$, and $\gt g_0\cong\gt g_0^*$.  Fix also the dual decomposition 
$\g^*=\gt g_0^*{\oplus}\gt g_1^*$. 
In order to avoid confusion,  let $\hat{\gt l}$ and $\hat{\gt c}$ denote 
the subspaces of $\g^*$  arising from $\gt l$ and $\gt c$, respectively, under
this identification.  
The orthogonal complements appearing below are taken with respect to the Killing form of $\gt g$. 
 Let $\tilde G=G_0\ltimes\exp(\gt g_1)$ 
be an algebraic group with $\Lie\tilde G=\g$.
The group $G_0$ is not necessary connected and therefore $\tilde G$ can 
also have  several connected components. However, note that each bi-homogeneous with respect to the decomposition $\gt g=\gt g_0{\oplus}\gt g_1$ 
component of $H\in\cS(\gt g)^{\gt g}$ is an invariant of $G_0$ and 
therefore in view of  Lemma~\ref{bullet}, $H^\bullet\in\cS(\g)^{\tilde G}$.  

\begin{rmk}\label{g.g.s.-hom}
In \cite{Dima06}, a good generating system (g.g.s.) consists of homogeneous polynomials 
by the definition. It is possible to show that if there is a g.g.s. in 
$\cS(\gt g)^{\gt g}$ with respect to a contraction $\gt g\ard\g$, then there are also
homogeneous polynomials forming a g.g.s.. We will not use and therefore will not prove 
this fact. In this and the following sections, all generating systems of invariants 
contain only homogeneous polynomials. 
\end{rmk}

\begin{ex}\label{e7-sym}
Take  $(\gt g,\gt g_0)=(E_7, \gt{so}_{12}{\oplus}\gt{sl}_2)$. Then $D_t=\dim\gt g_1=64$, and 
the generating homogeneous invariants 
$H_1,\ldots,H_7\in\cS(\gt g)^{\gt g}$ have degrees: $2,6,8,10,12,14,18$.
It is known that the restrictions of $H_1,H_2,H_3$, and $H_5$ to $\gt g_1^*$
generate $\cS(\gt g_1)^{\gt g_0}$, independently of the choice of 
$H_i$, see \cite{Helg}. We will show that there is a g.g.s.
in $Z\cS(\gt g)$ with respect to the contraction $\gt g\ard\g$.

Take any of the remaining three generators, say $H_j$. Assume that
$H_j^\bullet\in\cS(\gt g_1)$. Then it can be expressed as a polynomial $P$
in $H_i^\bullet$ with $i\in\{1,2,3,5\}$ and we can replace $H_j$ by
$H_j-P(H_1,H_2,H_3,H_5)$. Or rather assume from the beginning that
$\deg_t H_j <\deg H_j$. 

Next step is to show that $\deg_t H_j < \deg H_j-1$. 
Assume this not to be the case.  Restricting $H_j^\bullet$ to
$\hat{\gt l}{\oplus}\hat{\gt c}$, we get either zero or an $L$-invariant polynomial function
of bi-degree $(\deg H_j-1,1)$, in  other words,  a sum of $L$-invariants in
$\cS(\gt l)$ of an odd degree with coefficients from $\gt c$.
In this example $\gt l=\gt{sl}_2{\oplus}\gt{sl}_2{\oplus}\gt{sl}_2$
(see e.g. \cite[\S 5.4 and Table~9 in Ref. Chapter]{VO})
and all symmetric invariants have even degrees. 
This shows that $H_j^\bullet$ is zero on $\hat{\gt l}{\oplus}\hat{\gt c}$.
Clearly $H_j^\bullet$ is also zero on the 
$\tilde G$-saturation $\tilde G(\hat{\gt l}{\oplus}\hat{\gt c})$. 

Consider first the action of $\exp(\gt g_1)\subset \tilde G$. 
Note that $[\gt g,x]=\gt g_x^{\perp}$ for any $x\in\gt g$, and hence
$[\gt g_1,x]=\gt g_0\cap\gt g_x^{\perp}=\gt g_0\cap(\gt g_{0,x})^{\perp}$ for any $x\in\gt g_1$.
Now let $\hat s\in\hat{\gt c}$ be an element coming from some $s\in\gt c$. Then  
$$
\exp(\gt g_1)(\hat{\gt l}{\times}\{\hat s\})=\hat{\gt l}{\times}\{\hat s\}+\widehat{[\gt g_1,s]},
$$
where 
$\widehat{[\gt g_1,s]}$ is the annihilator  of $\gt g_{0,s}$
in $\gt g_0^*$.
Since $\gt g_{0,s}=\gt l$ for generic $s\in\gt c$, the saturation
$\exp(\gt g_1)(\hat{\gt l}{\oplus}\hat{\gt c})$ is a dense subset of $\gt g_{0}^*{\oplus}\hat{\gt c}$.
Applying $G_0$ to this subset, we get that $\overline{\tilde G(\hat{\gt l}{\oplus}\hat{\gt c})}=\g^*$
and therefore $H_j^\bullet=0$. Since the highest $t$-component of a non-zero polynomial  is non-zero, 
we get that $\deg_t H_j\le \deg H_j-2$.
Summing up, 
$$
\sum\limits_{i=1}^{7}\deg_t H_i\le \sum\deg H_i-6=70-6=64=D_t.
$$
Multiplying one of the $H_i$ by  a non-zero constant, we may assume that 
$H_1,\ldots, H_\ell$ satisfy the Kostant equality. 
Then, by Theorem~\ref{contr-deg}({\sf i}),({\sf ii}), $H_i^\bullet$ are algebraically 
independent, which means that  $H_i$ form a good generating system.
\end{ex}

In order to simplify calculations for other pairs, we prove a simple 
equality concerning ranks and dimensions.  
Recall that $\ell=\rk\gt g$.

\begin{lm}\label{rk-dim}
Let $\gt b_{\gt l}\subset\gt l$ be a Borel subalgebra.
Then $\dim\gt b=\dim\gt g_1+\dim\gt b_{\gt l}$.
\end{lm}
\begin{proof}
Clearly the subspace $\gt l{\oplus}\gt c$ contains a maximal torus of $\gt g$.
Therefore $\ell=\rk\gt l+\dim\gt c$. It is known that the dimension of a maximal $G_0$-orbit in $\gt g_1$
equals  $\dim\gt g_0-\dim\gt l$ on one hand, and $\dim\gt g_1-\dim\gt c$ on the other, 
see e.g. \cite[Proposition~9]{kr}.
Consequently, $\dim\gt g_0-\dim\gt l=\dim\gt g_1-\dim\gt c$.
Thereby we have
$$
\begin{array}{l}
\dim \gt b=(\dim\gt g+\ell)/2=(\dim\gt g_0+\dim\gt g_1+\ell)/2=\\
\quad = (\dim\gt g_1+\dim\gt l-\dim\gt c+\dim\gt g_1+\ell)/2=\dim\gt g_1+ \\
\qquad +(\dim\gt l-\dim\gt c+\rk\gt l+\dim\gt c)/2=\dim\gt g_1+\dim\gt b_{\gt l}.
\end{array}
$$
\end{proof}

The following assertion was predicted by D.\,Panyushev. 

\begin{thm}\label{goodpairs}
Let $(\gt g,\gt g_0)$ be a symmetric pair with $\gt g$ simple.
Suppose that the restriction map 
$\mK[\gt g]^G\to \mK[\gt g_1]^{G_0}$ is surjective. Then 
there is a good generating system $F_1,\ldots,F_\ell$ in 
$\cS(\gt g)^{\gt g}$ such that each $F_i$ is homogeneous and  $\cS(\g)^{\g}$ is freely generated by  $F_i^\bullet$.  
\end{thm}
\begin{proof}
According to \cite{Dima06}, there are 4 pairs to consider. 
For one of them the existence of a g.g.s. was 
established in Example~\ref{e7-sym}. 
Our next goal is to construct good generating systems for 3 classical pairs listed 
in Proposition~\ref{remaining}. 
We always assume that a set of generators $F_1,\ldots,F_\ell$ in 
$\cS(\gt g)^{\gt g}$ is normalised in order to satisfy the Kostant equality. 

For the first pair, with $\gt g=\gt{sp}_{2n+2m}$, 
we start with a set of generating invariants $\{H_1,\ldots,H_\ell\}\subset \cS(\gt g)^{\gt  g}$,
where each  $H_i$ is the sum of all 
principal $2i$-minors (this is also a coefficient of the characteristic polynomial). 
As can be readily seen from the block structure of this symmetric pair (Figure~\ref{symm-sp}), $\deg_t H_i\le 4m$ for all $i$. 
\begin{figure}[htb]
{\setlength{\unitlength}{0.023in}
\begin{center}
\begin{picture}(50,50)(0,0)
\put(0,0){\line(1,0){50}}
\put(0,0){\line(0,1){50}}
\put(50,0){\line(0,1){50}}
\put(0,50){\line(1,0){50}}
\put(0,20){\line(1,0){50}}
\put(30,0){\line(0,1){50}}

\put(10,35){$\gt{sp}_{2n}$}
\put(12,8){$\gt g_1$}
\put(33,7){$\gt{sp}_{2m}$}
\put(36,35){$\gt g_1$}
\end{picture}
\end{center}}
\caption{Symmetric decomposition of $\gt{sp}_{2n+2m}$}\label{symm-sp}
\end{figure}
To be more precise, for $1\le i\le m$, the highest $t$-components $H_i^\bullet$ 
lie in $\cS(\gt g_1)$ and form a generating set in $\mK[\gt g_1^*]^{G_0}$. 
Therefore set $F_i:=H_i$ for $i\le m$. 
Here $\gt l=(\gt{sl}_2)^m\oplus\gt{sp}_{2n-2m}$ and all symmetric $\gt l$-invariants 
are of even degrees. Applying the same trick as in Example~\ref{e7-sym},  
we can modify $H_j$ with $m<j\le 2m$ to $F_j$ in such a way that 
$\deg_t F_j\le 2j-2$. Remaining generators stay as they are, $F_i=H_i$ for 
$i>2m$. Summing up 
$$
\sum\limits_{i=1}^{\ell} (\deg F_i-\deg_t F_i) \ge 2m + \sum\limits_{j=1}^{n-m} 2j = 
\dim\gt b_{\gt l}\,.
$$
Making use of Lemma~\ref{rk-dim} and again of the equality $\sum\deg F_i=\dim\gt b$, we get that $\sum\deg_t F_i\le \dim\gt g_1=D_t$. 

\vskip0.5ex

For the second pair, with $\gt g=\gt{so}_{2\ell}$, we have $\gt l=(\gt{sl}_2)^{\ell/2}$,  then $\ell$ is even, and $\gt l=(\gt{sl}_2)^{[\ell/2]}\oplus\gt{so}_2$, then $\ell$ is odd. 
In case $\ell$ is even, we argue as in Example~\ref{e7-sym}. 
Choose homogeneous generators $F_i\in\cS(\gt g)^{\gt g}$ such that the highest components $F_1^\bullet,\ldots,F^\bullet_{\ell/2}$  form a generating set
in $\mK[\gt g_1^*]^{G_0}$, and $\deg_t F_i\le \deg F_i-2$ for  $i>\ell/2$.  
Then, taking into account Lemma~\ref{rk-dim}, we get
$$
\sum\limits_{i=1}^{\ell}\deg_t F_i \le \sum\limits_{i=1}^{\ell}\deg F_i - \ell =
\dim\gt b-\dim\gt b_{\gt l}=\dim\gt g_1=D_t\,.
$$

The case of odd $\ell$ is more interesting. 
We begin with a set of generating invariants 
$\{H_1,\ldots,H_\ell\}\subset \cS(\gt g)^{\gt  g}$,
where each $H_i$ with $i\ne \ell$ is the sum of all 
principal $2i$-minors and $H_\ell$ is the pfaffian, in particular, 
$\deg F_\ell=\ell$ is odd. One can realise $\gt{so}_{2\ell}$ as a set of
$2\ell{\times}2\ell$ matrices skew-symmetric with respect to the anti-diagonal. 
Then elements of $\gt g_1$ have block structure as shown in Figure~\ref{symm-so}. 
\begin{figure}[htb]
{\setlength{\unitlength}{0.023in}
\begin{center}
\begin{picture}(40,40)(0,0)
\put(0,0){\line(1,0){40}}
\put(0,0){\line(0,1){40}}
\put(40,0){\line(0,1){40}}
\put(0,40){\line(1,0){40}}
\put(0,20){\line(1,0){40}}
\put(20,0){\line(0,1){40}}

\put(8.5,28){$0$}
\put(8,8){$C$}
\put(29,8){$0$}
\put(27,28){$B$}
\end{picture}
\end{center}}
\caption{$\gt g_1$ for $(\gt{so}_{2\ell},\gt{gl}_{\ell})$. Here the matrices $B$ and $C$ are skew-symmetric with respect to the anti-diagonal.}\label{symm-so}
\end{figure}
This implies that all bi-homogenous (in $\gt g_0$ and $\gt g_1$) 
components of $H_i$ with $i<\ell$ have even degrees in $\gt g_1$ (and in $\gt g_0$).   

The highest $t$-components $H_i^\bullet$ with $2i<\ell$ form a generating set  
in $\mK[\gt g_1^*]^{G_0}$. Therefore we put $F_i:=H_i$ for these $i$. 
Each $H_j$ with $(\ell/2) < j < \ell$ can be modified to $F_j$ with 
$\deg_t F_j\le 2j-2$. And, finally, since $\det\xi=0$ for all $\xi\in\gt g_1$, 
we have $\deg_t H_\ell\le \ell-1$. Set $F_\ell:=H_\ell$. 
Then 
$$
\sum\limits_{i=1}^{\ell}\deg_t F_i \le \sum\limits_{i=1}^{\ell}\deg F_i - (\ell-1) -1 =
\dim\gt b-\dim\gt b_{\gt l}=\dim\gt g_1=D_t\,,
$$
where again we have used  Lemma~\ref{rk-dim}.

\vskip0.5ex

For the third pair, with $\gt g=\gt{sl}_{2n}$, 
we have $\gt l=(\gt{sl}_2)^{n}$. Here everything works exactly as in Example~\ref{e7-sym}.
We take homogeneous invariants $F_i$ with $\deg F_i=i+1$. Then 
$F_i^\bullet$ with $1\le i<n$ form a generating set in $\mK[\gt  g_1^*]^{G_0}$ and, 
modifying $F_j$ if necessary,
we can assume that $\deg_t F_j^\bullet\le \deg F_j-2$ for $j\ge n$. 
In view of Lemma~\ref{rk-dim},
$$
\sum\limits_{i=1}^{\ell}\deg_t F_i \le \sum\limits_{i=1}^{\ell}\deg F_i - 2n =
\dim\gt b-\dim\gt b_{\gt l}=\dim\gt g_1=D_t\,.
$$

For all three series we have constructed $F_i\in\cS(\gt  g)^{\gt g}$ such that 
$\sum\limits_{i=1}^{\ell}\deg_t F_i \le D_t$. By Theorem~\ref{contr-deg}({\sf i}),({\sf ii}), 
the polynomials $F_i$ form a  g.g.s..
Since in addition all $F_i$ are homogeneous here as well as in Example~\ref{e7-sym},   
the polynomials  
$F_i^\bullet$ generate $\cS(\g)^{\g}$ by \cite[Theorem~4.2(i)]{Dima06} or 
by Theorem~~\ref{contr-deg}({\sf iii}), if one recalls that $\g$ has the ``codim-2" property
\cite[Theorem~3.3.]{Dima06}.
\end{proof}
 
 
\begin{cl}[{cf. \cite[Theorem~4.2(ii)]{Dima06}}]\label{symm-Kost}
Let $(\gt g, \gt g_0)$ be a symmetric pair such that Theorem~\ref{goodpairs}
applies. Then the Lie algebra $\g$ is of Kostant type.
\end{cl}
 %
 
\subsection{Poisson semicentre}\label{sub:tube}

\begin{df}\label{semi}
Let $\gt q$ be a Lie algebra. Then 
an elements $H\in\cS(\gt q)$ is called a {\it semi-invariant} if 
$\{\xi,H\}\in\mK H$ for all $\xi\in\gt q$.
We let $\cS(\gt q)_{\rm si}$ denote the 
$\mK$-algebra generated by semi-invariants. This algebra is also called 
the Poisson semicentre of $\cS(\gt q)$. 
\end{df}

One of the easy to deduce properties of the semi-invariants is that 
$\{\cS(\gt q)_{\rm si},\cS(\gt q)_{\rm si}\}=0$,
see e.g. \cite[Section~2]{oomsb}.
Recently Poisson semicentres were studied in  \cite{oomsb} and 
\cite{jsh}. In particular, \cite{oomsb} proves a degree inequality for Lie algebras 
$\gt q$ such that $Z\cS(\gt q)$ is a polynomial ring and 
$Z\cS(\gt q)=\cS(\gt q)_{\rm si}$. Here we show that some $\Z_2$-contractions 
$\g$ of simple Lie algebras also satisfy the second property.   
If $G_0$ is semismple, then $\g$ has no non-trivial characters and 
clearly $Z\cS(\g)=\cS(\g)_{\rm si}$. 

Until the end of this section we assume that $G_0$ has a non-trivial connected 
centre. 
Since $\gt g$ is simple, the centre of $G_0$ is $1$-dimensional, 
see e.g. \cite[Section~6]{Dima06} (on a classification free basis this fact follows from 
a description of the finite order automorphisms of $\gt g$ in terms of  Kac diagrams).
Let $G_0'$ be the derived group of $G_0$ and 
$\gt g_0'=[\gt g_0,\gt g_0]$ the derived Lie algebra. 
By an elementary observation that $\gt g_0{\oplus}[\gt g_0,\gt g_1]$ 
is an ideal of $\gt g$, one proves the equality $\g'=\gt g_0'\ltimes\gt g_1$.

Recall that a symmetric space (or a symmetric pair) can be 
either of tube type, meaning $\mK[\gt g_1]^{\gt g_0'}\ne\mK[\gt g_1]^{\gt g_0}$,
or non-tube type. 
For example, $(\gt{so}_{2\ell},\gt{gl}_\ell)$ is of tube type if and only if $\ell$ is even. 
There are many  characterisations
of symmetric pairs of tube type. 
If $(\gt g, \gt g_0)$ 
is of tube type, then there are  more semi-invariants than symmetric $\g$-invariants,
because $\cS(\gt g_1)^{\gt g_0'}\subset\cS(\g)_{\rm si}$.
That case may be worth of investigating. 
Here we deal with symmetric spaces of non-tube type.

The following observation helps to treat semi-direct products
(cf.  \cite{r} or \cite[Proposition~5.5]{Dima07}). 

\begin{lm}\label{stab}
Let $\gt q=\gt f\ltimes V$ be a semi-direct product of a Lie algebra $\gt f$ and an Abelian ideal $V$.
Take $x=a+b\in\gt q^*$ with $a(V)=0=b(\gt f)$. Let $\gt a={\rm Ann}(\gt f{\cdot}b)$ be a subspace of 
$(V^*)^*=V$. 
Then $\gt a\lhd \gt q_x$ and $\gt q_x/\gt a\cong (\gt f_b)_{\tilde a}$,
where $\tilde a$ is the restriction of $a$ to $\gt f_b$.
\end{lm}
\begin{proof}
Note that $V{\cdot}b$ is zero on $V$, because $[V,V]=0$, and therefore $\gt q_x\subset\gt f_b\ltimes V$. 
It is also quite clear that $V{\cdot}b\subset {\rm Ann}(\gt f_b{\oplus}V)$. Hence 
$\gt q_x\subset  (\gt f_b)_{\tilde a}\ltimes V$. By the dimension reasons, 
$V{\cdot}b = {\rm Ann}(\gt f_b{\oplus}V)$ and for each $\xi\in(\gt f_b)_{\tilde a}$, there is $\eta\in V$ such that 
$\eta{\cdot}b=\xi{\cdot}a$. It remains to notice that $\gt q_x\cap V= {\rm Ann}(\gt f{\cdot}b)$.
\end{proof}

\begin{prop}\label{non-tube}
Suppose that $(\gt g, \gt g_0)$ is a symmetric pair of non-tube type.
Then $\cS(\g)_{\rm si}=Z\cS(\g)$.
\end{prop}
\begin{proof}
Here we consider the connected groups $\tilde G^{\circ}$ and 
$(\tilde G')^{\circ}=(G_0')^{\circ}\ltimes\exp(\gt g_1)$. 
Each character of $\tilde G^{\circ}$ is trivial on 
$(\tilde G')^{\circ}$, hence $\cS(\g)_{\rm si}\subset \cS(\g)^{\g'}$.
(In fact, $\cS(\g)_{\rm si} =\cS(\g)^{\g'}$).
Next we take $H\in \cS(\g)^{\g'}$ and show that it is an invariant 
of $\tilde G^{\circ}$.

Since $\gt g_0'$ is semisimple, $\mK(\gt g_1)^{\gt g_0'}$ is the quotient field 
of $\mK[\gt g_1]^{\gt g_0'}$. By Rosenlicht's theorem, generic orbits of an algebraic group, in our case $(G_0')^{\circ}$, 
are separated by rational invariants. Thereby 
the equality $\mK[\gt g_1]^{\gt g_0'}=\mK[\gt g_1]^{\gt g_0}$ implies that 
$G_0^{\circ}$ and 
$(G_0')^{\circ}$ have the same generic orbits in $\gt g_1$ and $\gt g_1^*$. 
On the Lie algebra level this means that $\gt g_0=\gt g_0'+\gt g_{0,b}$ for generic $b\in\gt g_1$. 

Suppose that $x=\hat a+\hat b\in\g^*$, where $\hat a$ and $\hat b$ correspond to generic $a\in{\gt l}$ and $b\in{\gt c}$, respectively. In view of Lemma~\ref{stab} and the fact that $\gt l=\gt g_{0,b}$ is reductive,  
$\g_x=\gt l_a{\ltimes}([\gt g_0,b])^{\perp}$, where 
$([\gt g_0,b])^{\perp}=\{\eta\in\gt g_1\mid [b,\eta]\in \gt g_0^{\perp}\}$
(the orthogonal complement is taken with respect to the Killing form of $\gt g$). 

Since $\gt g_0=\gt g_0'+\gt l$ and $\gt l_a$ contains 
the centre of $\gt l$, we have also $\gt g_0=\gt g_0'+\gt l_a$ and $\g=\g'+\g_x$. 
%
%
This leads to the equalities $\g'{\cdot}x=\g{\cdot}x$ and $\dim\tilde G'x=\dim\tilde Gx$. 
In addition, $(\tilde G')^{\circ}$ is a normal subgroup of $\tilde G^{\circ}$. Consequently, 
$(\tilde G')^{\circ}x=\tilde G^{\circ}x$. This equality holds on 
an open subset of $\tilde G^{\circ}(\hat{\gt l}{\oplus}\hat{\gt c})$, 
which is a dense subset of $\g^*$, because $\overline{{\tilde G}(\hat{\gt l}{\oplus}\hat{\gt c})}=\g^*$, as we know from Example~\ref{e7-sym}, and 
$\g^*$ is irreducible. 
Thus, $H$ is constant on a generic $\tilde G^{\circ}$-orbit and hence 
$H\in Z\cS(\g)$.    
\end{proof}

\section{Applications to E.\,Feigin's contraction}  

In this section, $\gt g=\Lie G$ is a simple Lie algebra of rank $\ell$,
$B\subset G$ is a Borel subgroup, and $\gt b=\Lie B$ is a Borel subalgebra.
We keep the assumption that $\mK=\overline{\mK}$. 
Fix a decomposition $\gt g=\gt b{\oplus}\gt n^-$, where 
$\gt n^-$ is the nilpotent radical of an opposite Borel, and consider 
a one-parameter contraction of $\gt g$ given by this decomposition. 
For the resulting Lie algebra $\g$, we have $\g=\gt b\ltimes\gt n^-$, 
where $\gt n^-$ is an Abelian ideal. 
This contraction was recently introduces by  E.\,Feigin  in  \cite{feigin1}. 
His motivation came from some problems in representation theory \cite{ffp}. 
Degenerations of flag varieties of $\gt g$ related to the contraction 
$\gt g\ard\g$ were further studied in \cite{feigin2} and \cite{fmp}.

Let $\{\alpha_1,\ldots,\alpha_\ell\}$ be a set of the simple roots and $e_i,f_i$ 
corresponding elements of the Chevalley basis. 
Set $\gt g_{\rm reg}=\{x\in\gt g\mid \dim\gt g_x=\ell\}$, 
where $\gt g_x$ is the stabiliser in the adjoint representation,  
$\gt n_{\rm reg}:=\gt n\cap\gt g_{\rm reg}$. 
If $x\in\gt n_{\rm reg}$,  then $\gt n_x=\gt b_x=\gt g_x$ 
and $Bx$  is a  dense open orbit in $\gt n$.
Hence $\gt n_{\rm reg}$ is a single $B$-orbit. 
The complement of this orbit was described by Kostant
 \cite[Theorem~4]{ko63} and 
$\gt n\setminus\gt n_{\rm reg}=\bigcup\limits_{i=1}^{\ell} \gt D_i$, where each 
$\gt D_i$ is a linear subspace of dimension $\dim\gt n-1$ orthogonal to $f_i$. 
We will also need an interpretation of $\gt D_i$ as closures of orbital varieties. 
It is a classical fact that for each nilpotent orbit $Ge\subset \gt g$, all irreducible components of $Ge\cap\gt n$ are of dimension $\frac{1}{2}\dim Ge$. In particular, 
$\gt n\setminus\gt n_{\rm reg}=\overline{{\mathcal O}^{\rm sub}\cap \gt n}$, 
where ${\mathcal O}^{\rm sub}$ is the unique nilpotent $G$-orbit in $\gt g$ of dimension $\dim\gt g-\ell-2$. 

Making further use of the Killing form $(\,\,,\,)$ of $\gt g$, we identify $(\gt n^{-})^*$ with the nilpotent radical $\gt n\subset\gt b$ 
and fix the dual decomposition 
$\g^*=\gt b^*{\oplus}\gt n^{\rm ab}$, where $\phantom{i}^{\rm ab}$
indicates that $\gt n^{\rm ab}$ is a space of the linear functions 
on an Abelian ideal. Let also $\gt n_{\rm reg}^{\rm ab}$ be  
a subset of $\gt n^{\rm ab}\subset\g^*$ corresponding to $\gt n_{\rm reg}$.
We identify $\gt D_i$ with subsets of $\gt n^{\rm ab}$ using the same letters for them.
 
Next statement was first proved in \cite{py-feig}.
 
\begin{lm}[cf. Lemma~\ref{stab}]\label{ind-f}
We have $\ind\g=\ell$.
\end{lm}
\begin{proof}
Clearly, $\rk\pi$ cannot get larger after a contraction, therefore 
$\ind\g\ge \ell$. 
On the other hand, take $x\in\gt n_{\rm reg}^{\rm ab}$ and extend it 
to a linear function on $\g^*$ by putting $x(\gt b)=0$.
Then $\g_x=\gt b_x=\gt n_x$ and it has dimension $\ell$.
Thus $\ind\g=\ell$.
\end{proof}

Another result of \cite{py-feig}, Theorem~3.3, states that $\cS(\g)^{\g}$ is freely 
generated by some polynomials $\widehat P_i$ (with $1\le i\le\ell$).
The construction of these polynomials $\widehat P_i$ starts with a system of homogeneous generators $F_i$ of $\cS(\gt g)^{\gt g}$ with 
$\deg F_i\le \deg F_{i+1}$. It is also shown that $\widehat P_i=F_i^\bullet$, 
\cite[Theorem 3.9]{py-feig}. 
We assume that $F_i$ are normalised to satisfy the Kostant equality.  

\begin{lm}\label{feig-kost}
Let $F_i$ be as above. Then $F_i^\bullet$ satisfy the Kostant equality with $\tilde\pi$ and therefore $\g$ is a Lie algebra of Kostant type. Besides, $\deg_t F_i=\deg F_i-1$ for all $i$. 
\end{lm}
\begin{proof}
Recall that ${\mathcal S}(\gt n^-)^{\gt b}=\mK$.
If $\deg _tF_i=\deg F_i$, i.e., $F_i^\bullet\in \cS(\gt n^-)$, 
then also $F_i^\bullet\in {\mathcal S}(\gt n^-)^{\gt b}$. A contradiction. 
Hence $\deg_t F_i\le \deg F_i-1$ for each $i$ and 
$$
\sum \deg_t F_i \le \dim\gt b-\ell=\dim\gt n=D_t\,.
$$
By Theorem~\ref{contr-deg}({\sf i}),({\sf ii}), $\deg_t F_i=\deg F_i-1$, the polynomials 
$F_i^\bullet$ are algebraically independent and satisfy the Kostant equality 
with $\tilde\pi$.  Since, according to \cite[Section~3]{py-feig}, $F_i^\bullet$ generate $\cS(\g)^{\g}$, the Lie algebra $\g$ is of Kostant type.  
\end{proof}

Actually, the bi-degrees of $F_i^\bullet$ with respect to the decomposition $\gt g=\gt b{\oplus}\gt n^-$
have been already found in \cite{py-feig}.

\begin{rmk}
Lemma~\ref{feig-kost} implies that $\cj(F_1^\bullet,\ldots,F_\ell^\bullet)=\mathrm{Sing}\,\tilde\pi$. Therefore $\mathrm{Sing}\,\tilde\pi$ contains a divisor 
whenever $\gt g$ is not of type $A$, see \cite[Th.~4.2\&Prop.~4.3]{py-feig}. 
This means that outside of type $A$ we get curious examples of Lie algebras of Kostant type that does not 
have the ``codim-2" property.
\end{rmk}

Next we turn our attention to the subset $\mathrm{Sing}\,\tilde\pi=\g^*_{\rm sing}$. 
The proof of Lemma~\ref{ind-f} shows that $\gt b^*{\times}\gt n_{\rm reg}^{\rm ab}\subset \g^*_{\rm reg}$
and therefore $\g^*_{\rm sing}\subset \gt b^*{\times}(\bigcup\limits_{i=1}^{\ell} {\gt D_i})$, where 
the subspaces $\gt D_i$ are regarded as subsets of $\gt n^{\rm ab}$.

\begin{df}\label{def-fund}
Let $\gt q$ be an $n$-dimensional Lie algebra with $\ind\gt q=\ell$ 
and $\pi$ its Lie-Poisson tensor. 
Then 
we will say that  a polynomial $p$ is a {\it fundamental semi-invariant} of $\gt q$, if
$\Lambda^{(n-\ell)/2}\pi=p R$ with $R\in W^{n-\ell}$ (notation as in Section~\ref{PS}) and 
the zero set of $R$ in $\gt q^*$  
has codimension grater than or equal to $2$. 
\end{df}

In \cite{oomsb}, {\it the} fundamental semi-invariant is defined as 
the greatest common divisor of the $\rk\pi{\times}\rk\pi$ (here $\rk\pi=n-\ell$)
minors in the matrix of $\pi$. Our polynomial is a square root of that 
one (up to a non-zero scalar) and is a scalar multiple of the fundamental semi-invariant 
in the sense of \cite[Section~4.1]{jsh}. 


Let 
$\delta$ be the highest root, $e_\delta$ a highest root vector, and
$r_i=[\delta:\alpha_i]$ the $i$-th coefficient in the decomposition of $\delta$, 
i.e.,  $\delta=\sum r_i\alpha_i$.

As is well-known, the highest degree of a homogeneous generator, 
under our assumptions, $\deg F_\ell$, equals $1+\sum r_i$.
Since 
$F_\ell^\bullet$ has weight zero and is of degree $1$ in $\gt b$ and  ($\deg F_\ell-1$) in 
$\gt n^-$, up to a scalar multiple $F_\ell^\bullet = e_\delta \prod f_i^{r_i}$.
This is also proved in \cite[Lemma~4.1]{py-feig}.

Set $p:=\prod f_i^{r_i-1}$. Note that in type $A$ we have $r_i=1$ for all $i$ and hence $p=1$. Here we generalise a result of \cite{py-feig}, Proposition~4.3, stating that in type $A$ the singular set $\mathrm{Sing}\,\tilde\pi$
contains no divisors. 

\begin{thm}\label{fund-f}
Let $\g$ be Feigin's contraction of a simple Lie algebra $\gt g$. Then $p=\prod f_i^{r_i-1}$ is a fundamental semi-invariant of $\g$.
\end{thm}
\begin{proof}
Set $\cF=dF_1^\bullet\wedge\ldots\wedge dF_\ell^\bullet$. 
Consider also a differential $1$-form   
$$
L=\left(\prod_{i=1}^{\ell} f_i\right)de_\delta+e_\delta\sum_{i=1}^{\ell} r_i f_1\ldots f_{i-1}f_{i+1}\ldots f_{\ell}df_i
$$
and set $R:=dF_1^\bullet\wedge\ldots\wedge dF_{\ell-1}^\bullet\wedge L$.
Note that $dF_\ell^\bullet=a p L$ with $a\in\mK^{^\times}$ and therefore $\cF=a p R$.
In view of the Kostant equality for $F_i^\bullet$ established in Lemma~\ref{feig-kost}, we have to show that the zero set of $R$
contains no divisors. 

Clearly, the zero set of  $R$ is contained in $\g^*_{\rm sing}$ and we have to prove that 
$R$ is non-zero on each irreducible divisor in $\g^*_{\rm sing}$. 
As was already mentioned, 
the proof of Lemma~\ref{ind-f} together with a result of Kostant \cite[Theorem~4]{ko63} imply that 
$\g^*_{\rm sing}\subset \gt b^*{\times}(\bigcup\limits_{i=1}^{\ell} {\gt D_i})$, where 
${\gt D}_i$ are the components of $\overline{{\mathcal O}^{\rm sub}\cap \gt n}$. 
Fix $i\in\{1,\ldots,\ell\}$. There is an element $e=e(i)$ in ${\gt D}_i$ such that $e\in {\mathcal O}^{\rm sub}$ and 
$(f_j,e)\ne 0$ for all $j\ne i$. Take this $e$ and add to it $b\in\gt b^*$ such that 
$b(e_\delta)\ne 0$ forming a linear function $x=b+e$ on $\g$.
Evaluating $L$ at $x$ we get $L_x=a'df_i$ with $a'\in\mK^{^\times}$. 
The goal is to prove that $d_xF_j^\bullet$ with $1\le j <\ell$ and $L_x$ are linear 
independent. To this end we calculate $d_xF^\bullet_j$.

Each $d_xF_j^\bullet$ can be considered as an element of $\g$ and therefore 
$d_xF_j^\bullet=\xi_j+\eta_j$ with $\xi_j\in\gt b$ and $\eta_j\in\gt n^{-}$. 
Since each $F_j^\bullet$ has degree $1$ in $\gt b$, 
$\xi_i=d_eF_i$, where $e$ is regarded as an element of $\gt g\cong\gt g^*$.  
By e.g. \cite[Sect.\,8.3,\,Lemma\,1]{slodowy} or
\cite[Theorem~10.6]{Dima07}, $d_eF_j$ with $j\le \ell$ generate a subspace 
of dimension $\ell-1$. For $d_xF_{\ell}^\bullet$ we have two possibilities, either 
it is zero (if $r_i>1$), or proportional to $df_i$. In any case, $\xi_\ell=0$. 
Therefore $\xi_1,\ldots,\xi_{\ell-1}$ are linear independent and 
clearly $d_xF^\bullet_j$ with $1\le j<\ell$ and $L_x$ together generate a subspace 
of simension $\ell$. This proves that $R$ is non-zero on each $\gt b^*{\times}\gt D_i$. 
Therefore $\dim\{\xi\in\g^*\mid R_\xi=0\}\le\dim\g-2$
and we are done. 
\end{proof}

\subsection{Proper semi-invariants}
A semi-invariant is said to be {\it proper} if it is not an invariant. 
The Lie algebra $\g$ possesses 
proper symmetric semi-invariants, for example $e_\delta$. Therefore describing
$\cS(\g)_{\rm si}$ is an interesting task. 

Set $H_i=F_i^\bullet$ for $1\le i <\ell$; $H_i=f_j$, where $j=i-\ell+1$ for 
$\ell\le i<2\ell$; and $H_{2\ell}=e_\delta$. Clearly all these 
functions are semi-invariants of $\g$. We will show that they generate 
$\cS(\g)_{\rm si}$. 

Let $\gt h\subset\gt b$ be a Cartan subalgebra of $\gt g$, 
$U\subset B$ the unipotent radical, and $\g'$ the derived algebra of $\g$. 
We have $\g'=\gt n{\ltimes}\gt n^{-}$. Note that the Lie algebra $\g'$ has only trivial 
characters. 

\begin{lm}\label{0}
We have
$\cS(\g)_{\rm si} = Z\cS(\g')$. In particular, $\cS(\g)_{\rm si}$
is a subalgebra of $\cS(\g')$. 
\end{lm}
\begin{proof}
Suppose that $\cS(\g)_{\lambda}$ is an eigenspace of $\g$ 
corresponding to a character $\lambda\in\g^*$ and  $\cS(\g)_{\lambda}\ne 0$.
Let $\g^\lambda\subset\g$ be the kernel of $\lambda$. 
Then, by a result of Borho \cite[Satz~6.1]{Bln}, 
$\cS(\g)_{\rm si}\subset \cS(\g^\lambda)$ (see also 
\cite[Lemme~4.1]{rv} or \cite[Sect.~1.2]{jsh}). 
Since $f_1,\ldots,f_\ell\in \gt n^{-}\subset\g$ are 
semi-invariants of the weights $-\alpha_i$ ($1\le i \le \ell$), 
we conclude that $\cS(\g)_{\rm si} \subset \cS(\g')$.
Next, $\g'$ has no non-trivial characters and the action of 
$\gt h$ on $Z\cS(\g')$ is diagonalisable. 
Thus indeed $\cS(\g)_{\rm si} = Z\cS(\g')$.
\end{proof}

Lemma~\ref{0} shows that $\g'$ is the canonical truncation of $\g$ in the following sense. 
For any algebraic finite-dimensional Lie algebra $\gt q$, there exists a unique subalgebra 
$\gt a$ such that $\cS(\gt q)_{\rm si}=Z\cS(\gt a)$ \cite{Bln}. This $\gt a$ is said to be the truncation of $\gt q$.


\begin{lm}\label{first}
Let $H_i$ be as above. Then 
the polynomials $H_i$ are algebraically independent. 
Besides, $\ind\g'=2\ell$.
\end{lm}
\begin{proof}
Let $\hat e\subset\gt n^{\rm ab}$ be a linear function coming from a regular nilpotent 
element $e\in\gt n$. We extend it to a function on $\g^*$ by setting $\hat e(\gt b)=0$. 
Then $d_{\hat e}F_i^\bullet=d_e F_i$. Moreover, $d_e F_{\ell}=e_\delta$ up to a constant. 
Therefore $d_e F_1,\ldots,d_eF_{\ell-1}$ and $e_\delta=de_\delta$ generate $\gt n_e$, 
a subspace of dimension $\ell$. The other polynomials $H_j$ ($\ell\le j<2\ell$) are linear independent elements of 
$\gt n^-$. Thus $d_{\hat e}H_i$ are linear independent and  the first statement is proved. 

According to the index formula of Ra{\"i}s \cite{r} (cf. Lemma~\ref{stab}), 
$$
\ind\g'=(\dim\gt n^{\rm ab}-\dim U\hat e)+\ind\gt n_{\hat e}=2\ell.
$$
Alternatively, one can use \cite[Lemma~3.7.]{oomsb}, which calculates the index of a truncated Lie algebra.
In our case it reads $\ind\g'=\ind\g+(\dim\g-\dim\g')$. 
\end{proof}

\begin{thm}\label{semi-f}
Let $\g$ be Feigin's contraction of $\gt g$. Then 
$\cS(\g)_{\rm si}$ is generated by the polynomials $H_i$ defined above.  
\end{thm}
\begin{proof}
First, we extract the ``$\gt h$-part'' out of $\tilde\pi$ and its powers. Let 
$\tilde\pi'$
be the Poisson tensor of $\g'$ and $\omega'$ a volume form 
on $(\g')^*$. 
Then $\tilde\pi=\tilde\pi'+R_{\gt h}$, where 
$R_{\gt h}$ is a sum $\sum [h_i,y_j]\partial_{h_i}\wedge\partial_{y_j}$ over a basis 
$h_1,\ldots,h_\ell$ of $\gt h$ and a basis of $\g'$. 
Since $\dim\gt h=\ell$, for $k>\ell$, we have $\Lambda^k R_{\gt h}=0$. 
Taking into account that also $\Lambda^k \tilde\pi'=0$ for $k>(n-3\ell)/2$ (Lemma~\ref{first}),
we get the equality
$$
\Lambda^{(n-\ell)/2}\tilde\pi=\left(\Lambda^{(n-3\ell)/2}\tilde\pi'\right)\wedge\left( \Lambda^{\ell} R_{\gt h}\right).
$$  
Therefore a fundamental semi-invariant of $\g'$ is a divisor of 
$p=\prod f_i^{r_i-1}$. 

Set $\cH:=dH_1\wedge\ldots\wedge dH_{2\ell}$. By Lemma~\ref{first}, $\cH\ne 0$. 
Note also that by the same lemma, $\ind\g'=2\ell$.
Therefore, 
applying Lemma~\ref{sleva-sprava} to $\cS(\g')$, we get non-zero coprime $q_1,q_2\in\cS(\g')$
such that 
$$
q_1(\cH/\omega')=q_2\Lambda^{(n-3\ell)/2}\tilde\pi'.
$$
Since the polynomials $q_1$ and $q_2$ are coprime, $q_1$ must be a divisor of $p$ as well.
In particular, $\deg q_1\le \deg p$. 
Next we compute and sum the degrees of all objects involved in the equality 
$$
\begin{array}{l}
\deg q_2+ (n-3\ell)/2=\deg q_1+ \deg\cH\le \deg p +\left(\sum\limits_{i=1}^{\ell-1}\deg F_i - \ell +1\right) = \\
\qquad \deg p+ ((n+\ell)/2-\deg F_{\ell}-\ell+1)= 
(n+\ell)/2-(\ell+1)-\ell+1=(n-3\ell)/2\,.
\end{array}
$$
This is possible only if $q_2\in\mK$ and $q_1=p$ (up to a scalar multiple). 
Thus $p(\cH/\omega')=a\Lambda^{(n-3\ell)/2}\tilde\pi'$
with $a\in\mK^{^\times}$. Moreover, $p$ is a fundamental 
semi-invariant of $\tilde\pi'$ and therefore the Jacobian locus $\cj(H_1,\ldots,H_{2\ell})$ 
of $H_i$ 
does not contain divisors.  We have 
$\dim\cj(H_1,\ldots,H_{2\ell})\le \dim\g-2$ and 
all polynomials $H_i$ are homogeneous. This allows us  to use the characteristic zero version of 
a result of Skryabin, see \cite[Theorem~1.1]{ppy}, stating that  here any $H\in \cS(\g')$ that is 
algebraic over a subalgebra generated by $H_i$ is contained in that subalgebra. 
Since $\mathrm{tr.\,deg}\,Z\cS(\g')\le 2\ell$, we conclude that 
$Z\cS(\g')$ is generated by $H_1,\ldots,H_{2\ell}$. 
Now the result follows from Lemma~\ref{0}.
\end{proof}

\subsection{Subregular orbital varieties}\label{orbital}
Irreducible components of $\overline{{\mathcal O}^{\rm sub}\cap\gt n}$ 
are called {\it subregular orbital varieties}. They have played a major r{\^o}le
in the proof of Theorem~\ref{fund-f} and we know that each of them is a linear 
space $\gt D_i$. Every $\gt D_i$ is also the nilpotent radical of a minimal parabolic 
subalgebra $\gt p_{\alpha_i}$. An interesting question is whether $B$ acts on $\gt D_i$ with an open orbit. 
This problem was addressed and solved in \cite{g-co}. 
As it turns out, our results complement and simplify some of the arguments 
in \cite{g-co}.

Let $\tilde G=B\ltimes\exp(\gt n^{-})$ be an algebraic group 
with $\Lie\tilde G=\g$. Let also $\mathcal{O}_i$ be an irreducible component of 
$\mathcal{O}^{\rm sub}\cap\gt n$ lying in $\gt D_i$. Note that $\mathcal{O}_i$ is a dense open subset 
of $\gt D_i$.

\begin{lm}\label{good-Di}
Suppose that $r_i=[\delta:\alpha_i]=1$. Then there is a dense open 
$B$-orbit in $\gt D_i$.
\end{lm}
\begin{proof}
According to Theorem~\ref{fund-f}, if $r_i=1$, then the intersection 
$(\gt b^*{\times}\gt D_i)\cap\g^*_{\rm reg}$ is non-empty, and hence it is 
a non-empty open subset of $\gt b^*{\times}\gt D_i$. 
Thereby there is a regular $x$ in $\gt b^*{\times}\mathcal{O}_i$. 
Let $\hat e$ be the $\gt n^{\rm ab}$-component of this $x$. In other words,  
$x\in \gt b^*{\times}\{\hat e\}$, where $\hat e\in\gt n^{\rm ab}$ comes from a subregular nilpotent element  $e\in \gt n$.
Next we compute the codimension of $\tilde Gx$.  
This can be done in the spirit of the Ra\"is formula for the index of a semi-direct product 
\cite{r}, see also 
\cite[Proposition~5.5]{Dima07} and Lemma~\ref{stab}. And the result is that 
\begin{equation}\label{eq-ind}
\dim\g-\dim\tilde Gx=\dim\gt n-(\dim\gt b-\dim\gt b_e)+\ind \gt b_e=\dim\gt b_e-\ell+\ind\gt b_e\,.
\end{equation}

Suppose that $Be$ is not dense in $\gt D_i$. Then $\dim Be<(\dim\gt n-1)$ and 
therefore $\dim\gt b_e\ge\ell+2$ implying $\gt b_e=\gt g_e$. 
Since $\ind\gt g_e\ge\ell$ by Vinberg's inequality, \cite[Corollary~1.7]{dima2}, we get 
that $\dim\g-\dim\tilde Gx\ge (\ell+2)$ and therefore $x\in\g^*_{\rm sing}$.
This contradiction proves the lemma.
\end{proof}  

In type $A$ all $r_i$ are equal to $1$ and all components $\gt D_i$ have open $B$-orbits. This result was 
first obtained by J.A.\,Vargas \cite{var}. 

\begin{ex}\label{type-D}
Suppose that $\gt g=\gt{so}_{2\ell}$ with $\ell>3$ and $e\in\gt g$ is a subregular nilpotent element.
Then $e$ is given by a partition $(2\ell-3,3)$, odd powers $e^{2k+1}$ of 
the underlying matrix are  elements of $\gt g$, 
and 
$\gt g_{e}$ has a basis 
$$
e,e^3,\ldots,e^{2\ell-5},\xi_1,\xi_2, \xi_3,\eta
$$ 
with the non-trivial commutators:
$[\xi_1,\xi_2]=e^{2n-5}$, $[\xi_1,\eta]=\xi_2$, and $[\xi_2,\eta]=\xi_3$.
(The structure of $\gt g_e$ is described, for example, in \cite[Section~1]{surpr}.) 
It is not difficult to see that $\gt g_e$ does not contain a commutative subalgebra 
of codimension $1$. Coming back to  equation~\eqref{eq-ind}, 
we get that in case $\dim\gt b_e=\ell+1$, 
the stabiliser $\gt b_e$ is not commutative and  thereby $\ind\gt b_e\le\ell-1$. 
As a consequence, $\dim\tilde Gx\ge \dim\g - \ell$. 
Thus in type $D$ there is an open $B$-orbit in $\gt D_i$ if and only if 
$r_i=1$.  
\end{ex}
 
Calculating centralisers $\gt g_e$ of  subregular nilpotent elements
in type $E$, on GAP or by hand, 
one can show that $\gt g_e$  does not contain 
an Abelian subalgebra of codimension $1$. 
Together with  Theorem~\ref{fund-f} and equation~\eqref{eq-ind},
this fact provides an additional explanation for 
\cite[Theorem~2.4(a)(i)]{g-co}. That result states that 
for $\gt g$ simply laced, $\gt D_i$ contains an open $B$-orbit if and only if $r_i=1$.

\begin{rmk}
Actually,  Theorem~2.4 of \cite{g-co} asserts that there is a finite number 
of $B$-orbits in $\mathcal{O}_i$, if $r_i=1$. As is explained in the Introduction of 
\cite{g-co}, this is equivalent to the existence of an open orbit. Note also that 
\cite{g-co} proves the existence of an open $B$-orbit by giving its representative in 
each particular case. Besides, results for the exceptional Lie algebras rely on 
GAP calculations of S.\,Goodwin \cite{sg}. 
\end{rmk}

 \vskip1ex
 
There is another 
interesting related question, asked by D.\,Panyushev. When is the 
stabiliser $B_e$ of a generic $e\in\gt D_i$ 
Abelian? The centraliser of a nilpotent element is Abelian only when the 
element is regular, for a conceptual proof of this fact see 
\cite[Theorem~3.3]{dima2}. In particular, 
$\gt g_e$ is not Abelian for a subregular nilpotent element $e$. 
This implies that $\gt b_e$ can be Abelian only if 
$\dim\gt b_e<\ell+2$ and there is a dense $B$-orbit in $\gt D_i$.
On the other hand, for an Abelian Lie algebra, $\ind\gt b_e=\dim\gt b_e=\ell+1$ and 
therefore $r_i$ must be larger than $1$. 
In  the simply laced case $\gt g_e$ does not contain Abelian
subalgebras of codimension $1$. For the remaining Lie algebras, 
\cite[Theorem~2.4.(a)(ii)]{g-co} provides the following answer. 

\begin{prop}\label{Ab}
The stabiliser in $B$ of a generic $e\in\gt D_i$ is  
Abelian if and only if 
\begin{itemize}
\item $\gt g$ is of type $B_{\ell}$ and $i>1$, or
\item $\gt g$ is of type $C_{\ell}$ with $\ell>1$ and $i=1$, or
\item $\gt g$ is of type $F_4$ and $i=4$, or
\item $\gt g$ is of type $G_2$ and $i=2$,
\end{itemize}
in the Vinberg-Onishchik numbering of simple roots \cite{VO}.
\end{prop}

\end{document}